\documentclass[article]{amsart}
\usepackage{stmaryrd}
\usepackage{amssymb}
\usepackage{enumitem}
\usepackage{soul}
\usepackage{mathabx}
\usepackage{amsmath}
\usepackage{amscd}
\usepackage{amsbsy}
\usepackage{cancel}
\usepackage{comment, enumitem}
\usepackage[matrix,arrow]{xy}
\usepackage{hyperref}

\hypersetup{
 colorlinks=true,
 linkcolor=violet,
 filecolor=blue,
 citecolor=brown,
 urlcolor=cyan,
 pdftitle={Lambda-invariant stability},
 }
\usepackage{mathrsfs}
\usepackage{color}
\usepackage{mathtools,caption}
\usepackage{tikz-cd}
\usepackage{longtable}
\usepackage[utf8]{inputenc}
\usepackage[OT2,T1]{fontenc}
\usepackage{changepage}
\usepackage{commath}
\usepackage{geometry}

\newcommand\numberthis{\addtocounter{equation}{1}\tag{\theequation}}

\DeclareMathOperator{\ur}{ur}

\DeclareMathOperator{\Hom}{Hom}

\DeclareMathOperator{\Gal}{Gal}
\DeclareMathOperator{\rank}{rank}
\DeclareMathOperator{\ord}{ord}

\DeclareMathOperator{\trace}{tr}
\DeclareMathOperator{\corank}{corank}
\DeclareMathOperator{\Sel}{Sel}

\DeclareMathOperator{\GL}{GL}

\DeclareMathOperator{\cyc}{cyc}

\DeclareMathOperator{\cont}{cont}

\DeclareMathOperator{\Frob}{Frob}
\DeclareMathOperator{\image}{Im}

\newcommand{\sE}{\mathscr E}
\newcommand{\sD}{\mathscr D}
\newcommand{\sN}{\mathscr N}
\newcommand{\Q}{\mathbb Q}
\newcommand{\F}{\mathbb F}

\newcommand{\cF}{\mathcal F}
\newcommand{\cP}{\mathcal P}
\newcommand{\cQ}{\mathcal Q}
\newcommand{\ZZ}{\mathbb Z}

\newcommand{\cO}{\mathcal O}
\newcommand{\fp}{\mathfrak{p}}

\newcommand{\Zp}{\ZZ_p}
\newcommand*{\EC}{\mathsf{E}}
\newcommand{\pdiv}{\mid\!\mid}

\DeclareSymbolFont{cyrletters}{OT2}{wncyr}{m}{n}
\DeclareMathSymbol{\Sha}{\mathalpha}{cyrletters}{"58}
\DeclareMathSymbol{\Zh}{\mathalpha}{cyrletters}{"11}

\usepackage[OT2,T1]{fontenc}
\input{cyracc.def}

\newcounter{para}

\newtheorem*{Theorem*}{Theorem}
\newtheorem*{Notation*}{Notation}
\newtheorem*{Conj*}{Conjecture}
\newtheorem*{Reform*}{Reformulation}
\newtheorem*{Defi*}{Definition}
\newtheorem*{hypo*}{Hypothesis}
\newtheorem*{prop*}{Proposition}
\newtheorem*{Ques*}{Question}
\newtheorem{Th}{Theorem}[section]
\newtheorem{Lemma}[Th]{Lemma}

\newtheorem{prop}[Th]{Proposition}

\newtheorem{Defi}[Th]{Definition}

\theoremstyle{definition}
\theoremstyle{remark}
\newtheorem{Remark}[Th]{Remark}
\newtheorem*{Remark*}{Remark}
\newtheorem{Example}[Th]{Example}

\makeatletter
\theoremstyle{plain} 
\newtheorem*{intr@thm}{\intr@thmname}

\newtheorem*{c@njecture}{\conjn@name}
\newcommand{\myl@bel}[2]{%
  \protected@write \@auxout {}{\string \newlabel {#1}{{#2}{\thepage}{#2}{#1}{}} }%
  \hypertarget{#1}{}
    } 
\newenvironment{labelledconj}[3][]%
    {
        \def\conjn@name{#2}
        \begin{c@njecture}[{#1}]\myl@bel{#3}{#2}
    }
    {
        \end{c@njecture}
    }
\makeatother

\definecolor{cadmiumgreen}{rgb}{0.0, 0.42, 0.24}

\numberwithin{equation}{section}

\begin{document}
\title[]{$\lambda$-invariant stability in Families of Modular Galois Representations}

\author[J.~Hatley]{Jeffrey Hatley}
\address[J.~Hatley]{
Department of Mathematics\\
Union College\\
Bailey Hall 202\\
Schenectady, NY 12308\\
USA}
\email{hatleyj@union.edu}

\author[D.~Kundu]{Debanjana Kundu}
\address[Kundu]{Fields Institute \\ University of Toronto \\
 Toronto ON, M5T 3J1, Canada}
\email{dkundu@math.toronto.edu}

\date{\today}

\keywords{Iwasawa theory, modular forms, Selmer groups, Galois representations}
\subjclass[2020]{11R23 (primary); 11F11, 11R18 (secondary) }

\begin{abstract}
Consider a family of modular forms of weight 2, all of whose residual $\pmod{p}$ Galois representations are isomorphic.
It is well-known that their corresponding Iwasawa $\lambda$-invariants may vary.
In this paper, we study this variation from a quantitative perspective, providing lower bounds on the frequency with which these $\lambda$-invariants grow or remain stable.
\end{abstract}

\maketitle

\section{Introduction}
\label{S: Intro}

Fix an odd prime $p$ and an algebraic closure $\overline{\Q}$ of $\Q$.
Suppose
\[
\overline{\rho} \colon \Gal(\overline{\Q}/\Q) \longrightarrow \GL_2(\F_p)
\]
is an absolutely irreducible, modular Galois representation of weight $2$ and conductor $N$, where $p \nmid N$.
Then by definition, there is some modular newform $g$, also of weight $2$ and level $N$, giving rise to $\overline{\rho}$ in the sense of \S~\ref{Ss:Galois-reps}.
That is, Deligne's associated Galois representation $\rho_g$ has a reduction mod $p$ which is isomorphic to $\overline{\rho}$.
In fact, $N$ is the \emph{optimal level} of $\overline{\rho}$, meaning that its level is minimal among all of the many other modular forms which also give rise to $\overline{\rho}$.
Denote by $\mathcal{H}(g)$ the set of all weight 2 newforms $f$ of level prime to $p$ such that $\bar{\rho}_f \simeq \bar{\rho}$.
To emphasize that $g$ is a modular form of optimal level in this family, we write $\mathcal{H}(g)$ rather than $\mathcal{H}(\overline{\rho})$.
For every $f \in \mathcal{H}(g)$, the level of $f$ is well-understood; see the description in \S~\ref{subsec:CDT}.

An important observation, originally pointed out by R.~Greenberg and V.~Vatsal for \emph{elliptic curves} with good \emph{ordinary} reduction at $p$ in \cite{greenbergvatsal} and later generalized by M.~Emerton, R.~Pollack, and T.~Weston in \cite{epw} to Hida families of $p$-ordinary modular forms, is that if $\mu(g)=0$, then $\mu(f)=0$ for every $f \in \mathcal{H}(g)$.
However, it is \textit{not} necessarily the case that $\lambda(g)=\lambda(f)$.
Rather, $\lambda(f)$ can be calculated from the knowledge of $\lambda(g)$ and some local information about $\rho_g$ and $\rho_f$ (see \S~\ref{S: local factors and growth of lambda}).

The results of \cite{greenbergvatsal} have been extended to the case when $p$ is a prime of supersingular reduction and $a_p=0$ in \cite{kim09,kim13}.
Here, we write $a_p$ to denote the $p$-th Fourier coefficient of the modular form (associated with the elliptic curve).
In \cite{hatleylei2019}, the results of \cite{epw} have been extended to the non-ordinary setting, i.e., when $p \mid a_p$.
In view of these more recent results, we will not assume that $g$ is ordinary at $p$.
In the non-ordinary setting, we have to consider \emph{pairs} of Iwasawa invariants, but we will suppress this in the notation; see \S~\ref{Ss:nonord}.

\begin{Remark}
The variation of $\lambda$-invariants in $\mathcal{H}(g)$ is not an abstract possibility; rather, there are examples in the literature of families $\mathcal{H}(g)$ with unbounded $\lambda$-invariants.
This was first studied by Greenberg for elliptic curves with good ordinary reduction at $p$ in \cite[Corollary~5.6]{Greenberg}.
For some recent examples of papers studying this phenomenon, see for example \cite{matsuno2007,kim09,Ray-lambda}.
\end{Remark}

The goal of this paper is to quantitatively study the variation of $\lambda$-invariants in $\mathcal{H}(g)$.
Starting with a modular form $g$ of weight $2$ and optimal level $N$, we consider the set of positive integers
\[
\mathcal{M}:=\left\{ \mathrm{level}(f) \ | \ f \in \mathcal{H}(g)\right\}.
\]
This is an infinite set and has been studied extensively by H.~Carayol in \cite{Car89} and by F.~Diamond--R.~Taylor in \cite{DT94}.
Our results have a different flavour: we are interested in quantifying how often subsets of $\mathcal{M}$ satisfy additional arithmetic properties.
In particular, we prove the following results:
\begin{itemize}
\item There exists a positive density of $M \in \mathcal{M}$ such that there is an $f$ in $\mathcal{H}(g)$ with level$(f)=M$ and $\lambda(f)=\lambda(g)$ (see Theorem~\ref{thm: reformulation of lambda f equal lambda E case}).
\item There exists a positive density of $M \in \mathcal{M}$ such that there is an $f$ in $\mathcal{H}(g)$ with level$(f)=M$ and $\lambda(f)>\lambda(g)$ (see Theorem~\ref{thm: reformulation of lambda positive}).
\end{itemize}
Moreover, our results are explicit, in the sense that we can compute a formula, depending only on $p$, for the lower bounds for these densities.

\subsection*{Future directions}
In this paper, we make the following assumptions to facilitate the methods employed in \S~\ref{sec:distribution-of-fourier}:
\begin{itemize}
\item $\overline{\rho}$ takes values in $\GL_2(\F_p)$ rather than in $\GL_2(\F_{p^m})$ for some $m \geq 1$;
\item $\overline{\rho}$ is of weight $2$; and
\item $\mathrm{det}(\bar{\rho})$ is exactly the $p$-cyclotomic character.
\end{itemize}
The growth of local factors which contribute to $\lambda$-invariants in more general Hida families of modular forms is briefly studied in \cite[Section~5.2]{epw}.
It would be interesting to remove the assumptions listed above and extend our results to a more general setting.

As discussed in Remark \ref{rmk:how-many-modular-forms-at-M}, it is a difficult open problem to determine, for each $M \in \mathcal{M}$, \textit{how many} newforms of level $M$ belong to $\mathcal{H}(g)$.
Progress towards this problem, which is interesting and important on its own, would also sharpen the results of the present paper.

We provide lower bounds for the densities under consideration, and it would be very interesting to determine upper bounds as well.

\subsection*{Organization}
Including this introduction, the article has five sections.
Section~\ref{S: preliminaries} is preliminary in nature, where we introduce the basic notation and definitions pertaining to Galois representations and Iwasawa theory.
In Section~\ref{sec:distribution-of-fourier} we study the distribution of Fourier coefficients of modular forms using linear algebra.
In Section~\ref{S:Quantifying Carayol's Theorem}, we use the results and techniques introduced in the previous section to quantify a result of Carayol which is central to our paper.
Finally, in Section~\ref{S: local factors and growth of lambda} we prove our main results via a careful study of the variation-of-$\lambda$-invariant formulas proven in \cite{epw} and \cite{hatleylei2019}.
We illustrate our main results with concrete examples.

\subsection*{Acknowledgements}
We thank Anwesh Ray and Rylan Gajek-Leonard for their comments and suggestions.
DK is supported by a PIMS Postdoctoral fellowship.
We thank the referees for their valuable suggestions which greatly improved the final version of this paper.

\section{Preliminaries}
\label{S: preliminaries}

\subsection{Galois representations}
\label{Ss:Galois-reps}
In this section, we recall some standard facts which will be required for our discussion; for details we refer the reader to \cite{DS05}.
Fix a prime $p \geq 3$.
Let $f=\sum_{n \geq 0} a_n(f)q^n \in S_2(\Gamma_0(M),\omega)$ be a modular form of weight $2$ and level $M$, where $p \nmid M$.
Denote by $K$ the number field generated by its Fourier coefficients $\{ a_n(f)\}$.
Let $\mathfrak{p} \mid p$ be a prime of $K$ above $p$, and denote by  $k_\mathfrak{p}$ its residue field.
We will assume that $k_\mathfrak{p}$ is isomorphic to $\mathbb{F}_p$.

By the work of P.~Deligne \cite{deligne}, associated to the modular form $f$ is a Galois representation
\[
\rho_f \colon G_\Q \longrightarrow \GL_2(K_\mathfrak{p}).
\]
This representation is unramified outside of $Mp$, and for each prime $\ell \nmid Mp$ we have
\begin{itemize}
\item $\trace \rho_f(\Frob_\ell)=a_\ell(f)$
\item $\det \rho_E(\Frob_\ell)=\ell \omega(\ell)$.
\end{itemize}
More generally, $\det \rho_f = \chi_p \omega$, where $\chi_p$ is the $p$-cyclotomic character and $\omega$ is the nebentypus character of $f$.

One may choose an integral model for $\rho_f$, and then taking composition with the reduction $\mod \mathfrak{p}$ map yields a residual representation
\[
\overline{\rho}_f \colon G_\Q \longrightarrow \GL_2(k_\mathfrak{p}).
\]
The semi-simple reduction of $\overline{\rho}_f$ is well-defined up to conjugation, i.e., independent of the choice of integral model.
Throughout this paper, we will assume our representations are irreducible.

More precisely, we assume that there is a surjective Galois representation
\[
\overline{\rho} \colon G_\Q \longrightarrow \GL_2(\F_p)
\]
realized by a modular form $g \in S_2(\Gamma_0(N))$ of optimal level.
We require that this optimal-level newform $g$ satisfies the following two hypotheses:
\begin{equation}
\tag*{\textup{\textbf{(Hyp~optimal)}}}\label{assmpn: optimal}
\begin{minipage}{0.8\textwidth}
\begin{enumerate}[label=\textup{(}\roman*\textup{)}]
\item \textbf{(Hyp~large-image)}: $\overline{\rho}_g \text{ is surjective}$.
\item \textbf{(Hyp~det)}: $\det \overline{\rho}_g = \chi_p$.
\end{enumerate}
\end{minipage}
\end{equation}
The large image hypothesis is satisfied for all but finitely many $p$ (cf \cite[Theorem~2.2.2]{LoefflerGlasgow}).
If $f \in \mathcal{H}(g)$, then $f$ also satisfies \ref{assmpn: optimal}.

\subsection{The theorems of Carayol and Diamond-Taylor}
\label{subsec:CDT}
Let $g$ and $f$ be newforms of weight $2$ and levels  $N$ and $M$, respectively.
Suppose that $\overline{\rho}_{g}$ is irreducible and that $\overline{\rho}_{g} \simeq \overline{\rho}_f$.
The following theorem by Carayol (see \cite{Car89}) imposes necessary conditions on the level $M$.
By a theorem of Diamond and Taylor, these conditions are also sufficient, see \cite[Theorem~A]{DT94}.
We note that the latter result was refined in \cite{DT2} to include the case $p=3$.

\begin{labelledconj}{Carayol's Theorem}{Thm:Carayol}
Suppose $g$ and $f$ are as described above, and $\overline{\rho}_{g}$ is irreducible.
If $\overline{\rho}_{g} \simeq \overline{\rho}_f$, then
\[
M = N\prod_{\ell}\ell^{\alpha(\ell)} \quad,
\]
and for each $\ell$ with $\alpha(\ell)>0$, one of the following holds:
\begin{enumerate}[label = \textup(\arabic*\textup)]
\item $\ell \nmid N$, $\ell \left(\trace \overline{\rho}_{g}\left(\Frob_{\ell} \right) \right)^2 = \left( 1+\ell\right)^2 \det\overline{\rho}_{g}\left( \Frob_{\ell}\right)$ in $\overline{\F}_p$, and $\alpha(\ell)=1$;
\item $\ell \equiv -1 \pmod{p}$ and one of the following holds:
\begin{enumerate}[label = \textup(\alph*\textup)]
\item $\ell \nmid N$, $\trace \overline{\rho}_{g}\left(\Frob_{\ell} \right) \equiv 0$ in $\overline{\F}_p$, and $\alpha(\ell)=2$;
\item $\ell \pdiv N$, $\det \overline{\rho}_{g}$ is unramified at $\ell$, and $\alpha(\ell)=1$;
\end{enumerate}
\item $\ell \equiv 1 \pmod{p}$ and one of the following holds:
\begin{enumerate}[label = \textup(\alph*\textup)]
\item $\ell \nmid N$ and $\alpha(\ell)=2$;
\item $\ell\pdiv N$ and $\alpha(\ell)=1$ \emph{or} $\ell\nmid N$ and $\alpha(\ell)=1$
\end{enumerate}
\end{enumerate}
\end{labelledconj}

We make a careful study of the conditions listed above in Section \ref{S:Quantifying Carayol's Theorem}.

\subsection{Iwasawa theory}\label{Ss:Iwasawa-thy}
Fix a prime $p> 3$ and an embedding $\iota_p: \overline{\Q} \hookrightarrow \overline{\Q}_p$.
Let $\Q_{\cyc}/\Q$ denote the cyclotomic $\Zp$-extension of $\Q$ and $\Q_{(n)}$ denote the unique subextension of degree $p^n$ over $\Q$ in $\Q_{\cyc}$.

\subsubsection{$p$-adic Selmer groups}
\label{S:Selmer}
Let $f=\sum_{n \geq 0} a_n(f)q^n$ be a modular cuspform of weight $2$, nebentypus $\omega$, and level $N$ such that $p\nmid N$.
For the moment, we assume that $f$ is $p$-ordinary.
Denote by $K$ the number field generated by the Fourier coefficients $\{a_n(f)\}$.
Let $\fp \mid p$ be a prime in $K$ and write $\cO_{\fp}$ for the ring of integers of $K_{\fp}$ with uniformizer $\pi$.

Let $V$ be a $2$-dimensional vector space over $K_{\fp}$ with $G_{\Q}$-action via $\rho_f$.
Fix a $G_{\Q}$-stable $\cO_{\fp}$-lattice $T$ in $V$ and set $A=V/T$.
Since $p$ is an ordinary prime, there is a $G_{\Q_p}$-stable line $V_p \subseteq V$  such that the action of $G_{\Q_p}$ on $V_p$ is given by the product of $\chi_p \omega$ and an unramified character, and the action of $G_{\Q_p}$ on $V/V_p$ is unramified.
Next, we can define
\[
A_p = \image\left( V_p \longrightarrow A\right)
\]
such that the $G_{\Q_p}$-action on $A/A_p$ is unramified.
Let $v$ (resp. $\eta$) denote the unique prime above $p$ (resp. a prime not above $p$) in $\Q_{\cyc}$ and write $I_v$ (resp. $I_{\eta}$) to denote the inertia subgroup of $G_{\Q_{\cyc,v}}$ (resp. $G_{\Q_{\cyc,\eta}}$).
Define
\begin{align*}
H^1_{\ord}\left( \Q_{\cyc,v}, A\right) &= \ker\left(H^1\left( \Q_{\cyc,v},A\right) \longrightarrow H^1\left(I_v, A/A_p\right)\right)\\
H^1_{\ur}\left( \Q_{\cyc,\eta}, A\right) &= \ker\left(H^1\left( \Q_{\cyc,\eta},A\right) \longrightarrow H^1\left(I_{\eta}, A\right)\right).
\end{align*}
The \emph{Selmer group for $f$ over $\Q_{\cyc}$} is defined as the kernel of a global-to-local restriction map, i.e.,
\[
\Sel\left(\Q_{\cyc},A\right) = \ker\left(H^1\left( \Q_{\cyc},A\right) \longrightarrow \prod_{\eta\nmid p} \frac{H^1\left( \Q_{\cyc,\eta},A\right)}{H^1_{\ur}\left( \Q_{\cyc,\eta}, A\right)} \times \prod_{v\mid p}\frac{H^1\left( \Q_{\cyc,v},A\right)}{H^1_{\ord}\left( \Q_{\cyc,v}, A\right)}\right).
\]

\subsubsection{Iwasawa Invariants}
Let $\Gamma:=\Gal(\Q_{\cyc}/\Q)$ be topologically isomorphic to $\Zp$ and write $\cO$ for the ring of integers in a finite extension $\cF/\Q_{p}$.
Write $\Lambda=\cO\llbracket \Gamma \rrbracket$ for the completed group ring of $\Gamma$ with coefficients in $\cO$.
The \emph{Pontryagin dual} of a compact or discrete $\Lambda$-module $M$ is defined as
\[
M^\vee = \Hom_{\cont}\left(M, \cF/\cO \right).
\]

We now recall the \emph{Structure Theorem for finitely generated $\Lambda$-modules.}
If $M$ is a finitely generated $\Lambda$-module, then there is a pseudo-isomorphism (i.e., a $\Lambda$-module
homomorphism with finite kernel and cokernel)
\[
M \longrightarrow \Lambda^r \oplus  \bigoplus_{i=1}^t \Lambda/\pi^{m_i} \oplus  \bigoplus_{j=1}^s \Lambda/\fp_j^{n_j},
\]
where $r,s,t\geq 0$ (with the convention that if one of these is zero, the summand does not
appear), $m_i$ and $n_j$ are positive integers, and $\fp_j$ are (not necessarily distinct) height one prime ideals of $\Lambda$.
The integer $r$ is called the \emph{$\Lambda$-rank} of $M$.
We define the Iwasawa invariants as follows:
\begin{align*}
\mu(M) = \sum_{i=1}^t m_i \quad \text{ and } \quad
\lambda(M) = \sum_{j=1}^s \rank_{\cO}\left( \Lambda/\fp_j^{n_j}\right).
\end{align*}
For the discrete module $M^\vee$, set $\corank_{\Lambda}(M^\vee) =\rank_{\Lambda}(M)$, $\mu(M^\vee)=\mu(M)$ and $\lambda(M^\vee)=\lambda(M)$.
We say that $M$ (resp. $M^\vee$) is $\Lambda$-torsion (resp. $\Lambda$-cotorsion) if $r=0$.

\subsubsection{Iwasawa Theory of Selmer groups}
The absolute Galois group $G_{\Q}$ acts via conjugation on the $\cO$-submodule $H^1\left(\Q_{\cyc},A \right)$ with $G_{\Q_{\cyc}}$ acting trivially.
This action allows us to view the global cohomology group as a discrete $\Lambda$-module.
The Selmer group $\Sel\left(\Q_{\cyc},A\right)$ defined in Section~\ref{S:Selmer} is a $G_{\Q}$-stable $\cO$-submodule of $H^1\left(\Q_{\cyc},A\right)$ and is a discrete $\Lambda$-module, as well.

Let $S$ be a finite set of primes containing the archimedean primes, the prime $p$, and the primes where $A$ is ramified.
Denote by $\Q_S$ the maximal $S$-ramified extension of $\Q$.
Then, there is another way to define the Selmer group, which is as follows:
\[
\Sel\left(\Q_{\cyc},A\right) = \ker\left(H^1\left( \Q_S/\Q_{\cyc},A\right) \longrightarrow \prod_{\substack{\eta\mid q\\ q\in S\setminus \{p\}}} \frac{H^1\left( \Q_{\cyc,\eta},A\right)}{H^1_{\ur}\left( \Q_{\cyc,\eta}, A\right)} \times \prod_{v\mid p}\frac{H^1\left( \Q_{\cyc,v},A\right)}{H^1_{\ord}\left( \Q_{\cyc,v}, A\right)}\right).
\]
\cite[Proposition~3]{Gre89} asserts that $\Sel\left(\Q_{\cyc},A\right)$ is a cofinitely generated $\Lambda$-module.
Hence, its Iwasawa $\mu$ and $\lambda$ invariants are well-defined.
For ease of notation, we will henceforth write $\mu(f)$ and $\lambda(f)$ for the Iwasawa invariants of $\Sel\left(\Q_{\cyc},A\right)$.

\subsection{Non-ordinary case}\label{Ss:nonord} Now let us suppose that our modular form $f$ is non-ordinary at $p$, so that $p \mid a_p(f)$.
In this setting, it is appropriate to define a \emph{pair} of (signed) Selmer groups $\Sel^\pm\left(\Q_{\cyc},A\right)$.
In the case when $f$ corresponds to an elliptic curve $\EC_{/\Q}$ and $p$ is a prime of supersingular reduction, it follows from the Hasse--Weil bound that $a_p(f)=0$ whenever $p>3$.
These signed Selmer groups were first defined by S.~Kobayashi in \cite{kobayashi03}.
When $a_p(f)\neq 0$, the corresponding Selmer groups were studied by F.~Sprung in \cite{sprung09}.

To review the construction for more general modular forms, we refer the reader to \cite[\S~2]{hatleylei2019}.
In all of these cases, the construction is similar to that described in the previous subsections, but with modified local conditions at the primes above $p$.

Just as before, one can attach Iwasawa invariants $\mu^\pm(f)$ and $\lambda^\pm(f)$ via the Structure Theorem for the cofinitely-generated $\Lambda$-modules $\Sel^\pm \left(\Q_{\cyc},A\right)$.

\begin{Remark}
Apart from whether we must consider an ordinary Iwasawa invariant (i.e. $\lambda(f)$) or one of two signed invariants (i.e. $\lambda^+(f)$ or $\lambda^-(f))$, the arguments in the rest of the paper have no dependence on whether $f$ is $p$-ordinary or whether $f$ is non-ordinary.
Hence, for notational clarity, in the latter case we suppress the $\pm$ notation entirely.
\end{Remark}

\subsection{Cotorsion hypothesis} For the entirety of this article, and without further mention, we make the following assumption which is crucially required to apply the results of \cite{epw} (when $f$ is $p$-ordinary) and \cite{hatleylei2019} (when $f$ is non-ordinary):
\begin{equation}
\tag*{\textup{\textbf{(Hyp~cot)}}}\label{assmpn: cotorsion}
\Sel\left(\Q_{\cyc},A\right) \text{ is cotorsion over }\Lambda.
\end{equation}
When $f$ has good ordinary reduction at $p$, this hypothesis follows from the work of K.~Kato \cite{Kato}.
In the non-ordinary case, this is known in many cases, e.g. \cite[Theorem~7.3]{kobayashi03} for elliptic curves with $a_p(\EC)=0$, \cite[Theorem~1.2]{sprung09} for elliptic curves with $a_p(\EC)\neq0$, and  \cite[Theorem~6.4]{lei09} for modular forms when $a_p(f)=0$.
Note that we have already begun to suppress the $\pm$ notation when necessary.

\begin{Remark}
It would be possible to remove the assumption \ref{assmpn: cotorsion} using \cite[Corollary~3.4]{hatley-lei-jnt} if one could obtain a more explicit description of the terms called $c(A)$ and $c(B)$ in \textit{loc. cit.}
\end{Remark}

\section{Distribution of Fourier Coefficients}
\label{sec:distribution-of-fourier}
In this section, we want to understand the distribution of Fourier coefficients of non-CM modular forms.
Let $f \in S_2(\Gamma_0(N))$ and let $p \nmid N$ be a fixed prime such that $\overline{\rho}_{f}$ is surjective.
(We want to clarify that the $N$ here is \emph{not} the optimal one but rather a general notation for the discussion in this
section.)
Denote the $\ell$-th Fourier coefficient of $f$ by $a_\ell(f)$ and let $a$ be any fixed integer.
Then, the primary goal of this section will be to understand the following limits.
\begin{align}
\label{part 1 of conj}&\lim_{x \rightarrow \infty} \frac{\# \{\ell \text{ is a prime }<x : a_{\ell}(f) \equiv a \pmod{p}\}}{\pi(x)} \\
\label{part 2 of conj}&\lim_{x \rightarrow \infty} \frac{\# \{\ell \text{ is a prime }<x : a_{\ell}(f) \equiv \pm\left( \ell+1\right) \pmod{p}\}}{\pi(x)}.
\end{align}
Here, we use the standard notation $\pi(x)$ to denote the number of primes up to $x$.

First, we define subsets of invertible matrices with $\F_p$ coefficients.
\begin{Defi}
Let $m,n$ be two integers considered $\pmod{p}$.
Define
\[
C_{m,n} = \left\{\gamma \in \GL_2(\F_p) : \det(\gamma)=m \text{ and } \trace(\gamma)=n\right\}.
\]
\end{Defi}

The following notion of density will be important for our purposes.
\begin{Defi}
A set $\mathcal{S}$ of prime numbers is called a \emph{Chebotarev set} if there
is a Galois extension $K/\Q$ and a conjugacy-stable set $C \subseteq \Gal(K/\Q)$ such that $\mathcal{S}$ agrees with
the set $\{p : \Frob_p \in C\}$ up to a finite set.
Finite unions, finite intersections, and complements of Chebotarev sets are again Chebotarev.
\end{Defi}
The Chebotarev density theorem states
that if $\mathcal{S}$ arises from $K$ and $C$ as above, then the limit
\[
\mathfrak{d}(S) = \lim_{x\rightarrow\infty}
\frac{\# \mathcal{S} \cap [1, x]}{\pi(x)}
\]
exists and equals $\#C/[K : \Q]$.
The quantity $\mathfrak{d}(\mathcal{S})$ is called the \emph{density} of $S$.

We remind the reader that the residual Galois representation
\[
\overline{\rho}_{f}: \Gal\left( \Q(\overline{\rho}_{f})/\Q\right) \longrightarrow \GL_2(\F_p)
\]
satisfies the following two properties for $\ell \nmid N$:
\begin{itemize}
\item $\det\left(\overline{\rho}_{f}\left(\Frob_{\ell}\right)\right) \equiv \ell \pmod{p}$ \emph{and}
\item $\trace\left(\overline{\rho}_{f}\left(\Frob_{\ell}\right)\right) \equiv a_\ell(f) \pmod{p}$.
\end{itemize}
Upon restricting the domain to the finite Galois extension $\Q(\overline{\rho}_{f})/\Q$ cut out by $\overline{\rho}_f$, our surjectivity hypothesis implies $\overline{\rho}_{f}$ is an isomorphism.

\subsection*{Recollections from Linear Algebra}
Recall that matrices in the same conjugacy class have the same characteristic polynomial, i.e., they have the same determinant and the same trace.
The group $\GL_2(\F_p)$ can be divided into the following conjugacy classes:
\begin{itemize}
\item Let $\sD_{a,b}$ be the set of diagonalizable matrices with eigenvalues $a,b\in {\F_p}^\times$ with $a\ne b$.
There are $(p-1)(p-2)/2$ choices of $\sD_{a,b}$ and for each choice, $\#\sD_{a,b}=p(p+1)$.
\item Let $\sN_a$ be the set of non-diagonal matrices with one single eigenvalue $a\in {\F_p}^\times$.
There are $(p-1)$ choices for $\sN_a$ and for each choice, $\#\sN_a=p^2-1$.
\item Let $\sD_a=\left\{\begin{pmatrix}a&0\\0&a\end{pmatrix}\right\}$, $a\in \F_p^\times$.
Then, there are $(p-1)$ choices for $a$ and for each choice $\#\sD_a=1$.
\item Let $\sE_{\lambda}$ be the set of matrices whose eigenvalues are $\lambda$ and $\lambda'$, where $\lambda\in\F_{p^2}\setminus\F_p$ and $\lambda'$ is the conjugate of $\lambda$.
There are $p(p-1)/2$ choices for $\lambda$ and for each choice of $\lambda$, $\# \sE_{\lambda}=p^2-p$.
\end{itemize}

\subsection{Understanding \texorpdfstring{\eqref{part 1 of conj}}{}.}
\begin{Defi}
Let $p>2$ and let $f \in S_2(\Gamma_0(N))$ be a modular form such that $\overline{\rho}_f \colon G_\Q \rightarrow \GL_2(\F_p)$ is surjective.
Fix two integers $i,a\pmod{p}$.
Define $\cQ_{i,a}(f, p, x)$ to be the set of primes $\ell< x$ satisfying the following two conditions:
\begin{enumerate}[label =\textup(\roman*\textup)]
\item $\ell \equiv i \pmod{p}$ \emph{and}
\item $a_{\ell}(f) \equiv a \pmod{p}$
\end{enumerate}
\end{Defi}
Next, define
\[
\cQ_{a}\left( f, p, x\right): = \bigcup_{i=1}^{p-1} \cQ_{i,a}\left(f ,p, x\right) = \left\{\ell \text{ is a prime }<x : a_{\ell}(f) \equiv a \pmod{p}\right\}.
\]
Since each $\cQ_{i,a}\left(f,p, x\right)$ is disjoint, it is clear that
\[
\#\left\{\ell \text{ is a prime }<x : a_{\ell}(f) \equiv a \pmod{p}\right\} =
\sum_{i=1}^{p-1} \# \cQ_{i,a}(f, p, x).
\]

The sets $\cQ_{i,a}\left(f,p\right)= \lim_{x\rightarrow \infty}\cQ_{i,a}\left(f,p,x\right)$ are Chebotarev sets.
More precisely,
\begin{equation}
\label{eqn: Q to C}
\cQ_{i,a}(f, p) = \left\{\ell: \ell\nmid N \text{ and } \overline{\rho}_{f}\left(\Frob_{\ell}\right)\in C_{i,a}\right\}.
\end{equation}
Thus, to study the density arising in \eqref{part 1 of conj} it suffices to obtain information on the density of $C_{i,a}$ for all $1\leq i \leq p-1$.

\subsubsection{When $a=0$}
\label{a = 0}
The main goal of this section is to prove the following lemma
\begin{Lemma}
\label{Lemma: count a =0}
With notation introduced above,
\[
\frac{\sum_{i=1}^{p-1}\# C_{i,0}}{\# \GL_2\left( \F_p\right)} = \frac{p}{p^2 -1}.
\]
\end{Lemma}

\begin{proof}
It is easy to see that $\# \GL_2(\F_p) = p(p-1)^2(p+1)$.
Therefore, we only need to count the numerator, i.e., we need to count matrices $\gamma\in \GL_2(\F_p)$ with $\det \gamma =i$ and $\trace \gamma =0$.
The characteristic polynomial of each such matrix therefore is of the form
\[
X^2 + i \pmod{p}.
\]
For obtaining the count, we proceed by calculating the eigenvalues of these matrices.
Note that
\[
X^2 + i \equiv 0 \pmod{p}
\]
has a solution \emph{precisely} when $-i$ \emph{is a quadratic residue} $\pmod{p}$.
This happens for $\frac{p-1}{2}$ of the residue classes.
On the other hand, this equation \emph{does not} have a solution when $-i$ \emph{is not a quadratic residue} $\pmod{p}$, which also happens for $\frac{p-1}{2}$ of the residue classes.

Suppose that $-i$ is a quadratic residue $\pmod{p}$.
Then, the two eigenvalues must be different.
Such a matrix lies in the conjugacy class $\sD_{a,-a}$.
For each residue class, the contribution is $p(p+1)$.
Therefore, the total contribution is
\[
\frac{(p-1)p(p+1)}{2}.
\]
On the other hand, next suppose that $-i$ is not a quadratic residue $\pmod{p}$.
Then, the characteristic polynomial does not have a solution $\pmod{p}$.
Such a matrix lies in the conjugacy class $\sE_{\lambda}$.
For each residue class, the contribution is $p^2 -p$.
Therefore, the total contribution is
\[
\frac{(p-1)^2p}{2}.
\]

We can now compute the density
\begin{align*}
\frac{\sum_{i=1}^{p-1}\# C_{i,0}}{\# \GL_2\left( \F_p\right)} &= \frac{(p-1)p(p+1) + (p-1)^2p}{2p(p-1)^2(p+1)}\\
&= \frac{p}{(p-1)(p+1)} = \frac{p}{p^2 -1}.
\end{align*}
This completes the proof of the lemma.
\end{proof}

The next result is now immediate from the above discussion.

\begin{prop}\label{prop:al-zero}
Let $p$ be an odd prime and let $f \in S_2(\Gamma_0(N))$ be a modular form such that $\overline{\rho}_f \colon G_\Q \rightarrow \GL_2(\F_p)$ is surjective.
Then,
\[
\lim_{x \rightarrow \infty} \frac{\# \{\ell \text{ is a prime }<x : a_{\ell}(f) \equiv 0 \pmod{p}\}}{\pi(x)} = \frac{p}{p^2 -1}.
\]
\end{prop}

\subsubsection{When $a\neq 0$}
The main goal of this section is to prove that when $a\neq 0$,
\[
\frac{\sum_{i=1}^{p-1}\# C_{i,a}}{\# \GL_2\left( \F_p\right)} = \frac{p^2 - p -1}{(p-1)^2 (p+1)}.
\]

For the sake of clarity of exposition, we carry out the calculation in the special case that $a= \pm 1$.
As will be clear to the reader, the same proof goes through in the general case, as well.
\begin{Lemma}
With notation introduced above,
\[
\frac{\sum_{i=1}^{p-1}\# C_{i,\pm 1}}{\# \GL_2\left( \F_p\right)} = \frac{p^2 - p -1}{(p-1)^2 (p+1)}.
\]
\end{Lemma}

\begin{proof}
This time we are counting matrices $\gamma\in \GL_2(\F_p)$ with $\det\gamma = i$ and $\trace \gamma =\pm 1$
The characteristic polynomial of each such matrix therefore is of the form
\[
X^2 \mp X + i \pmod{p}.
\]
For obtaining the count, we proceed by calculating the eigenvalues of these matrices.
As before, we are interested in the solutions of
\[
X^2 \mp X + i \equiv 0 \pmod{p}.
\]
Since $p$ is an odd prime, this is the same as the solutions of the following equation which we obtain by completing the squares
\begin{equation}
\label{eqn: completing the square}
(2X \mp 1)^2 \equiv 1 - 4i \pmod{p}.
\end{equation}
Writing $Y = 2X \mp 1$, the above equation has a solution \emph{precisely} when $1-4i$ \emph{is a quadratic residue} $\pmod{p}$.
This happens for $\frac{p-1}{2}$ of the residue classes.
On the other hand, this equation \emph{does not} have a solution when $1-4i$ \emph{is not a quadratic residue} $\pmod{p}$, which also happens for $\frac{p-1}{2}$ of the residue classes.

Suppose that the $1-4i$ is a quadratic residue $\pmod p$.
Note that if $4i \equiv 1\pmod{p}$, then \eqref{eqn: completing the square} has a solution.
Moreover, the eigenvalues must repeat in this case.
The matrices lie in the conjugacy class $\sN_1$ or $\sD_1$.
The contribution from this conjugacy class is
\[
p^2 - 1 + 1 = p^2.
\]
In the remaining $\frac{p-1}{2} - 1 =\frac{p-3}{2}$ conjugacy classes, the two eigenvalues must be different.
Such a matrix lies in the conjugacy class $\sD_{a,b}$.
For each residue class, the contribution is $p(p+1)$.
Therefore, the total contribution of matrices is
\[
\frac{(p-3)p(p+1)}{2}.
\]
When the $1-4i$ is not a quadratic residue $\pmod{p}$, we have mentioned before that the characteristic polynomial does not have a solution $\pmod{p}$.
Such a matrix lies in the conjugacy class $\sE_{\lambda}$.
For each residue class, the contribution is $p^2 -p$.
Therefore, the total contribution is
\[
\frac{(p-1)^2p}{2}.
\]

We can now compute the density
\begin{align*}
\frac{\sum_{i=1}^{p-1}\# C_{i,\pm 1}}{\# \GL_2\left( \F_p\right)} &= \frac{2p^2 + (p-3)p(p+1) + (p-1)^2p}{2p(p-1)^2(p+1)}\\
&= \frac{2p + (p-3)(p+1) + (p^2-2p+1)}{2(p-1)^2(p+1)} \\
&= \frac{2(p^2 - p -1)}{2(p -1)^2 (p+1)} = \frac{p^2 - p -1}{(p-1)^2(p+1)}.
\end{align*}
This completes the proof of the lemma.
\end{proof}

The next result is now immediate from the above discussion.

\begin{prop}\label{prop:al-nonzero}
Let $p$ be an odd prime and let $f \in S_2(\Gamma_0(N))$ be a modular form such that $\overline{\rho}_f \colon G_\Q \rightarrow \GL_2(\F_p)$ is surjective.
Suppose that $a\not \equiv 0\pmod{p}$.
Then,
\[
\lim_{x \rightarrow \infty} \frac{\# \{\ell \text{ is a prime }<x : a_{\ell}(f) \equiv a \pmod{p}\}}{\pi(x)} = \frac{p^2 - p -1}{(p-1)^2(p+1)}.
\]
\end{prop}

\begin{Remark}
Adding up the proportions, we obtain
\[
\frac{p}{p^2 - 1} + (p-1)\frac{p^2 - p -1}{(p-1)^2 (p+1)} = \frac{p^2 -1 }{p^2 -1} = 1,
\]
as desired.
\end{Remark}

\subsection{Understanding \texorpdfstring{\eqref{part 2 of conj}}{}.}
\label{section: understanding part 2 of the conj}
\begin{Defi}
Let $p>2$ and let $f \in S_2(\Gamma_0(N))$ be a modular form such that $\overline{\rho}_f \colon G_\Q \rightarrow \GL_2(\F_p)$ is surjective.
Define $\cP_{i}(f, p,x)$ to be the set of primes $\ell<x$ satisfying the following two conditions:
\begin{enumerate}[label =\textup(\roman*\textup)]
\item $\ell \equiv i \pmod{p}$ \emph{and}
\item $a_{\ell} \equiv (i+1) \pmod{p}$.
\end{enumerate}
\end{Defi}
Then
\[
\left\{\ell \text{ is a prime }<x : a_{\ell}(f) \equiv \left( \ell+1\right) \pmod{p}\right\} = \bigcup_{i=1}^{p-1} \cP_i\left(f ,p, x\right).
\]
Since each $\cP_i\left( f,p, x\right)$ is disjoint, it is clear that
\[
\#\left\{\ell \text{ is a prime }<x : a_{\ell}(f) \equiv \left( \ell+1\right) \pmod{p}\right\} =
\sum_{i=1}^{p-1} \# \cP_{i}(f, p, x).
\]

In our situation, the sets $\cP_i\left(f,p\right) = \lim_{x\rightarrow \infty} \cP_i\left(f,p, x\right)$ are Chebotarev sets.
More precisely,
\begin{equation}
\label{eqn: P to C}
\cP_i(f, p) = \left\{\ell: \ell\nmid N \text{ and } \overline{\rho}_{f}\left(\Frob_{\ell}\right)\in C_{i,i+1}\right\}.
\end{equation}
We note for future reference that, upon replacing the condition $a_\ell \equiv (i+1) \pmod p$ with $a_\ell \equiv -(i+1) \pmod p$, one also obtains Chebotarev sets
\begin{equation*}
\cP_i^-(f, p) = \left\{\ell: \ell\nmid N \text{ and } \overline{\rho}_{f}\left(\Frob_{\ell}\right)\in C_{i,-(i+1)}\right\}.
\end{equation*}

Hence, to study the density arising in \eqref{part 2 of conj} it suffices to obtain information on the density of $C_{i,i+1}$ for all $1\leq i \leq p-1$.
Every element $\gamma$ in $C_{i,i+1}$ satisfies the following two equations:
\[
\begin{split}
\det (\gamma) & \equiv i \pmod{p}\\
\trace (\gamma) & \equiv i+1 \pmod{p}.
\end{split}
\]
The characteristic polynomial of interest is therefore
\[
X^2 - \trace(\gamma) X + \det(\gamma)I_2 =0.
\]
Suppose that the eigenvalues are $\lambda_1, \lambda_2$.
Then an invertible matrix $\gamma$ is in the set $C_{i,i+1}$ precisely when
\[
\begin{split}
\det (\gamma) &= \lambda_1 \lambda_2 \equiv i \pmod{p}\\
\trace (\gamma) &= -(\lambda_1 + \lambda_2) \equiv i+1 \pmod{p}.
\end{split}
\]
It is easy to check that the eigenvalues are $\equiv -1 \text{ or } -i\pmod{p}$.
This tells us that $\lambda_1, \lambda_2\in \F_p^\times$ and we only need to focus on the first three conjugacy classes.
When $i=1$, the element $\gamma$ belongs to the conjugacy class $\sN_1$ or $\sD_1$ because both the eigenvalues are equal.
Therefore,
\begin{equation}
\label{count C12}
\# C_{1,2}= \# \sN_{1} + \# \sD_{1} = (p^2-1) + 1 = p^2.
\end{equation}
On the other hand, when $i\neq 1$, the element $\gamma$ lies in the conjugacy class $\sD_{1,i}$.
Therefore,
\begin{equation}
\label{C i i+1 equation}
\# C_{i,i+1} = p(p+1).
\end{equation}
It follows that
\begin{equation*}\label{eq:Pi-sum}
\sum_{i=1}^{p-1}\frac{\#C_{i,i+1}}{\# \GL_2(\F_p)} = \frac{p^2 + (p-2)p(p+1)}{p(p-1)^2 (p+1)} = \frac{p^2 -2}{(p-1)^2(p+1)}.
\end{equation*}
By repeating the calculations done above, we also obtain that
\begin{equation}
\label{C i - i+1 equation}
\# C_{i,-(i+1)} = p(p+1).
\end{equation}
We have therefore proven the following result:
\begin{Lemma}
With notation as above,
\[
\sum_{i=1}^{p-1}\frac{\#C_{i,-(i+1)}}{\# \GL_2(\F_p)} = \frac{p^2 -2}{(p-1)^2(p+1)}.
\]
\end{Lemma}

\section{Quantifying \ref{Thm:Carayol}}
\label{S:Quantifying Carayol's Theorem}
\subsection{Simplifying the problem}
\label{reformulating carayol}
We fix a non-CM modular form $g \in S_2(\Gamma_0(N))$ satisfying \ref{assmpn: optimal} and \ref{assmpn: cotorsion}, fix a large $x$, and vary $\ell$ over all primes less than $x$.
Our first goal is to count how many $M$ may possibly arise from \ref{Thm:Carayol}.

We make the following observations which will simplify our counting problem:
\begin{enumerate}
\item Since we are only interested in asymptotics and $N$ is divisible by finitely many $\ell$, we can (and will) ignore the behaviour of $\ell\mid N$ in \ref{Thm:Carayol}.
\item The condition $\ell \left(\trace \overline{\rho}_{g}\left(\Frob_{\ell} \right) \right)^2 = \left( 1+\ell\right)^2 \det\overline{\rho}_{g}\left( \Frob_{\ell}\right)$ in $\overline{\F}_p$ that arises in \ref{Thm:Carayol} can be simplified in our situation due to our assumptions that $g$ has trivial nebentypus and $\overline{\rho}_g$ takes values in $\GL_2(\F_p)$.
In particular, as noted in Section~\ref{S: preliminaries}, we have
\begin{align*}
\trace \overline{\rho}_{g}\left(\Frob_{\ell} \right)&\equiv a_{\ell}(g) \pmod{p}\\
\det \rho_{g}\left(\Frob_{\ell} \right)&\equiv \ell \pmod{p}
\end{align*}
where the quantities above belong to $\F_p$, and so the aforementioned condition from \ref{Thm:Carayol} becomes
\[
a_{\ell}\equiv \pm (1+\ell)\pmod{p}.
\]
Similarly, since $\det \overline{\rho}_{g}$ is the $p$-cyclotomic character, the condition in \ref{Thm:Carayol} that $\det \overline{\rho}_{g}$ is unramified at $\ell$ is vacuous.
\end{enumerate}

For the rest of this section, we simplify the notation by omitting $g$, writing e.g. $a_\ell$ instead of $a_\ell(g)$.

The primes that \emph{do not} divide $N$ can be divided into the following \emph{disjoint} sets:
\begin{itemize}
\item Set 1: $\ell\nmid N$, $\ell \not\equiv \pm 1\pmod{p}$, and $a_{\ell}\equiv \pm (1+\ell)\pmod{p}$.
\item Set 1': $\ell\nmid N$, $\ell \not\equiv \pm 1\pmod{p}$, and $a_{\ell}\not\equiv \pm (1+\ell)\pmod{p}$.
\item Set 2: $\ell\nmid N$, $\ell \equiv -1 \pmod{p}$, and $a_{\ell} \equiv 0 \pmod{p}$.
\item Set 2': $\ell\nmid N$, $\ell \equiv -1 \pmod{p}$, and $a_{\ell} \not\equiv 0 \pmod{p}$.
\item Set 3: $\ell\nmid N$ and $\ell \equiv 1\pmod{p}$.
\end{itemize}
We introduce the following piece of notation which will be used in the remainder of the paper.
\begin{Notation*}
For a set of primes $\mathcal{S}$, we write $\mathcal{S}(x):= \mathcal{S}\cap [1,x)$.
\end{Notation*}

We make the following key observations for counting:
\begin{enumerate}
\item Denote Set 1 by $\mathfrak{S}_1$; and set $s_1(x) = \#\mathfrak{S}_1\cap [1,x)$.
For each prime in this set, there are two options: we can either pick it (in which case $\alpha(\ell)=1$) or not pick it (in which case $\alpha(\ell)=0$).
In short, there are $2^{s_1(x)}$ many ways of making a choice from this set.
\item Denote Set 2 by $\mathfrak{S}_2$; and set $s_2(x) = \#\mathfrak{S}_2\cap [1,x)$.
For each prime in this set, there are again three options: we can either pick it (in which case $\alpha(\ell)=1$ or $2$) or not pick it (in which case $\alpha(\ell)=0$).
So there are $3^{s_2(x)}$ ways of picking primes from this set.
\item Denote Set 3 by $\mathfrak{S}_3$; and set $s_3(x) = \#\mathfrak{S}_3\cap [1,x)$.
For each prime in this set there are 3 options.
We may not pick the prime (in which case $\alpha(\ell)=0$) or pick it once (in which case $\alpha(\ell)=1$) or pick it twice (in which case $\alpha(\ell)=2$).
So there are $3^{s_3(x)}$ ways of picking primes from this set.
\item Any prime in Set 1' or Set 2' \emph{does not} arise in \ref{Thm:Carayol}.
These primes always have $\alpha(\ell)=0$.
For the remainder of this section, we will call them \emph{discard sets}.
\end{enumerate}
Thus, the total number of possibilities of an integer $M$ in \ref{Thm:Carayol} is precisely the number of ways to pick a collection of primes up to $x$ from one of the sets above.
This number is
\begin{equation}
\label{count from Carayol}
2^{s_1(x)}\times 3^{s_2(x)+s_3(x)}.
\end{equation}
But, there is one case when we did not pick any primes $\ell$ from any of the sets.
In this case, $M=N$ and we ignore it.
For each of the other cases, \cite[Theorem~A]{DT94} (applied iteratively) asserts that there is \emph{at least one} honest modular form attached to the $M$ from \ref{Thm:Carayol}.
Our objective for the rest of this section is to describe the asymptotic behaviour of the functions $s_i(x)$.

\subsection{Studying the density of each set}
We now apply the results of Section \ref{sec:distribution-of-fourier} to the study of the functions $s_i(x)$.

\subsection*{Determining the density of the primes in Set 1}

\begin{Lemma}
\label{box 1 count}
Write $\pi(x)$ to denote the number of primes up to $x$.
Then
\[
\lim_{x\to \infty}\frac{s_1(x)}{\pi(x)} = \frac{2(p-3)}{(p-1)^3}.
\]
\end{Lemma}

\begin{proof}
Recall that the primes counted by $s_1(x)$ are those for which all of the following conditions hold:
\begin{enumerate}[label = (\roman*)]
\item $\ell \nmid N$;
\item $\ell \not\equiv \pm 1\pmod{p}$
\item $a_\ell \equiv \pm \left( \ell+1\right) \pmod{p}$.
\end{enumerate}
Then the results of Section \ref{section: understanding part 2 of the conj} yield
\begin{align*}
\lim_{x\to \infty}\frac{s_1(x)}{\pi(x)} &= \mathfrak{d}(\mathfrak{S}_1) \\
&= \sum_{\ell=2}^{p-2}\left(\mathfrak{d}(\cP_\ell(f, p)) + \mathfrak{d}(\cP_\ell^-(f, p))\right) \\
&=(p-3) \left(\frac{\#C_{\ell,\ell+1}}{\# \GL_2(\F_p)} + \frac{\#C_{\ell,-(\ell+1)}}{\# \GL_2(\F_p)}\right)  \\
&= \frac{2p(p-3)(p+1)}{p(p-1)^2 (p+1)}\\ &= \frac{2(p-3)}{(p-1)^2}
\end{align*}
as claimed.
\end{proof}

\subsection*{Determining the density of the primes in Set 2}

\begin{Lemma}
\label{box 2 count}
Write $\pi(x)$ to denote the number of primes up to $x$.
Then
\[
\lim_{x \to \infty}\frac{s_2(x)}{\pi(x)} = \frac{1}{(p-1)^2}.
\]
\end{Lemma}

\begin{proof}
Recall that the primes counted by $s_2(x)$ are those for which all of the following conditions hold:

\begin{enumerate}[label = (\roman*)]
\item $\ell \nmid N$;
\item $\ell \equiv -1 \pmod{p}$;
\item $a_{\ell}\equiv 0 \pmod{p}$.
\end{enumerate}
We are thus computing the density $\mathfrak{d}(\mathfrak{S}_2)$, for which we will rely on the calculations done in Section~\ref{a = 0}.
Since condition (ii) is equivalent to saying that $-\ell \equiv 1 \pmod{p}$ and $1$ is always a quadratic residue $\pmod{p}$, it follows from the counting argument used in the proof of Lemma~\ref{Lemma: count a =0} that the desired density is given by
\[
\left(\frac{\# C_{p-1,0}}{\# \GL_2(\F_p)}\right) = \frac{p(p+1)}{p(p-1)^2 (p+1)} =  \frac{1}{(p-1)^2}.
\]
(Here, we only have to keep track of the count in one residue class.)
\end{proof}

\subsection*{Determining the density of the primes in Set 3:}
For this set the counting is considerably simpler.
\begin{Lemma}
\label{box 3 count}
Write $\pi(x)$ to denote the number of primes up to $x$.
Then,
\[
\lim_{x \to \infty} \frac{s_3(x)}{\pi(x)} = \frac{1}{p-1}.
\]
\end{Lemma}

\begin{proof}
We need to count the number of primes satisfying the following two conditions:
\begin{enumerate}[label = (\roman*)]
\item $\ell \nmid N$;
\item $\ell \equiv 1 \pmod{p}$.
\end{enumerate}
 By Dirichlet's theorem on primes in arithmetic progressions, the density of such primes is exactly
\[
\frac{1}{p-1}.
\]
\end{proof}



\section{Local factors and the growth of \texorpdfstring{$\lambda$}{}-invariants}
\label{S: local factors and growth of lambda}
\subsection{Local factors}
\label{SS: local factors}
Retain the notation from Section~\ref{S: preliminaries}.
We have non-CM modular forms $g$ and $f$ of weight $2$ and levels $N$ and $M$, respectively, whose residual Galois representations are isomorphic.

\begin{Remark}\leavevmode
\begin{enumerate}
\item Despite our assumption that $\bar{\rho}_f$ takes values in $\GL_2(\F_p)$, $f$ need not correspond to an elliptic curve.
Recall the notation from $\S$~\ref{S:Selmer}.
In general, associated  to a weight $2$ modular form $f$ is an abelian variety over $\Q$ of degree $[K:\Q]$.
Thus, if $f \in \mathcal{H}(g)$, then there exists some prime $\mathfrak{p} \mid p$ with inertia degree $1$.

We explain this via an explicit example: let $g$ be the modular form corresponding to the elliptic curve $X_0(11)$.
Let $f \in S_2(\Gamma_0(143))$ be the modular form with LMDFB label \href{https://www.lmfdb.org/ModularForm/GL2/Q/holomorphic/143/2/a/b/}{143.2.a.b}.
This modular form has coefficients in the number field generated by the polynomial $x^4 - 4x^2 - x + 1$.
Let $\alpha$ denote a root of this polynomial.
The prime $\mathfrak{p}=(\alpha^3-4\alpha)$ lies above $3$ and has residue field isomorphic to $\F_3$.
Using the command \texttt{Reductions(f,3)} in Magma \cite{magma} and Sturm's Bound \cite{sturm}, one may compare the mod $3$ and mod $\mathfrak{p}$ $q$-expansions of $g$ and $f$, respectively, to verify that $f \in \mathcal{H}(g)$.
\item In a similar vein to the previous remark, recall that $f \in S_2(\Gamma_0(M),\omega)$.
By \cite[Theorem~A]{DT94}, it is possible that $f$ has nontrivial nebentypus even though $g \in S_2(\Gamma_0(N))$.
However, writing $\bar{\omega}$ for the $\pmod{p}$ reduction of $\omega$, we see that since $\bar{\rho}_f \simeq \bar{\rho}_g$, we must have that
\[
\det \bar{\rho}_f \simeq \chi_p \bar{\omega} \simeq \chi_p.
\]
Thus $\bar{\omega}$ is trivial and will not play a role in the constant term of $P_\ell(X)$.
\end{enumerate}
\end{Remark}

Recall that we assume $\overline{\rho}_{g}$ is irreducible; in this case, it is conjectured that $\mu(g)=0$ (e.g. when $g$ corresponds to an elliptic curve, this is Greenberg's Conjecture \cite[Conjecture 1.11]{Greenberg}).
In light of this well-known conjecture, we will henceforth assume\footnote{We remark that by \cite[Theorem~3.7]{KR21}, this is the generic case when $g$ corresponds to an elliptic curve $\EC_{/\Q}$.} the following
\begin{equation}
\tag*{\textup{\textbf{(Hyp~$\mu$)}}}\label{assmpn: mu}
\text{ The } \mu\text{-invariant associated to the Selmer group of } g \text{ is }0.
\end{equation}
Then by \cite[Theorem~2 and Theorem~4.3.3]{epw} (in the ordinary case) or \cite[Corollary 2.10 and Theorem~4.6]{hatleylei2019} (in the non-ordinary case)  we have that $\mu(f)=0$ and
\begin{equation}
\label{EPW thm 2}
\lambda(g) + \sum_{\ell \mid M} \delta(g,\ell) = \lambda(f) + \sum_{\ell \mid M} \delta(f,\ell),
\end{equation}
where the local factors $\delta$ are given by the size of some local Galois cohomology groups, as described in \cite[Proposition~2.4]{greenbergvatsal}.
More precisely, $\delta(\cdot,\ell)=s_\ell d_\ell$ where $s_\ell$ is the largest power of $p$ dividing $\frac{\ell^{p-1}-1}{p}$ (equivalently, it is the number of primes $\eta$ lying above $\ell$ in $\Q_{\cyc}$) and $d_\ell$ is described below.
In particular, the value of $s_\ell$ is the same for both $g$ and $f$.

To describe the value of $d_\ell$, we use the explanation given in \cite[Proposition~2.4]{greenbergvatsal} (see also \cite[\S~5.2]{epw} and \cite[\S~5]{hatleylei2019}).
We have translated their result to be consistent with the normalization of Frobenius characteristic polynomials that we are using in the rest of the paper.

Let $P_\ell(X)$ denote the reduction $\pmod{p}$ of the characteristic polynomial of $\Frob_\ell$ on $(V_f)_{I_\ell}$.
Then $d_\ell$ is the multiplicity of $1$ as a root of $P_\ell(X) \pmod{p}$.
In particular, $\delta(\cdot, \ell)>0$ precisely when $P_\ell(1)=0$.

When $\ell \mid M$, the corresponding characteristic polynomial is
\[
P_\ell(X)=\ell-a_\ell(f)X.
\]
Thus $\delta(f,\ell)>0$ if and only if $a_\ell(f) \equiv \ell$ mod $p$.
Note that in this case, $1$ is necessarily a simple root, so $d_\ell(f)=1$.
We thus have the following simple description: for $\ell \mid M$, we have
\begin{align*}
\delta(f,\ell)>0 &\Longleftrightarrow  d_{\ell} = 1 \\
&\Longleftrightarrow  \delta(f,\ell)=s_\ell \\
&\Longleftrightarrow a_\ell(f) \equiv \ell \pmod{p}. \numberthis \label{eqn:local-at-bad-primes}
\end{align*}

\begin{Remark}
We point out that, by the Local Langlands correspondence (see e.g. \cite[Proposition~2.8(2)]{LW}, the matrix-counting methods of Section~\ref{sec:distribution-of-fourier} still apply for these ramified primes.
\end{Remark}

Now we consider the local factors for $g$.
If $\ell \mid N$, then the analysis is identical to the above.
Otherwise, $\ell \nmid N$ and the characteristic polynomial is
\[
P_\ell(X)=\ell - a_\ell(g)X + X^2.
\]
This implies that
\begin{itemize}
\item $1$ is a root if and only if $1+\ell \equiv a_\ell(g) \pmod{p}$, and
\item $1$ is a \emph{simple} root if and only if $1+\ell \equiv a_\ell(g) \pmod{p}$ and $\ell \not\equiv 1 \pmod{p}$.
\end{itemize}

Thus, rewriting \eqref{EPW thm 2} gives
\begin{equation}\label{eq:compare-lambdas}
\lambda(g) =
\lambda(f) + \sum_{\ell \mid M}[\delta(f,\ell)-\delta(g,\ell)].
\end{equation}
Since $g$ is fixed, both sides of the above equation are constant as we vary $f$.
In particular,
\begin{equation}
\label{eqn: strict inequality}
\lambda(f)>\lambda(g) \Longleftrightarrow \sum_{\ell \mid M}\delta(g,\ell) > \sum_{\ell \mid M}\delta(f,\ell).
\end{equation}

\subsection{How often is \texorpdfstring{$\lambda(f)=\lambda(g)$}{}?}\label{subsec:equal-lambda}
The main goal of this section is to study the stability of the $\lambda$-invariant among the modular forms in $\mathcal{H}(g)$ using the explicit description of their levels provided by \ref{Thm:Carayol} and \cite[Theorem~A]{DT94}.
For the sake of simplicity, we make the following assumption.
\begin{equation}\tag*{\textup{\textbf{(Hyp~min)}}}\label{assmpn: minimal lambda for E}
\begin{minipage}{0.85\textwidth}
$\lambda(g)$ is minimal among all $\lambda(f)$  for $f \in \mathcal{H}(g)$.
\end{minipage}
\end{equation}

\begin{Remark}
This condition is satisfied for example when $\lambda(g)=0$, which is the generic case when $g$ corresponds to an elliptic curve $\EC_{/\Q}$ with Mordell--Weil rank equal to $0$ (cf \cite[Theorem~3.7]{KR21}).
It is also satisfied when $g$ is $p$-ordinary by \cite[Corollary 2]{epw}.
\end{Remark}

\begin{Lemma}
\label{lemma: delta equals 0 implies lower bound}
Let $g \in S_2(\Gamma_0(N))$ be a newform satisfying \ref{assmpn: optimal} and that $\mu(g)=0$.
Let $M$ be an integer arising from \ref{Thm:Carayol} and $f \in \mathcal{H}(g)$ be an associated modular form.
If $\delta(g,\ell) =0$ for every $\ell\mid M$, then $\lambda(f)\leq \lambda(g)$.
Further, if $g$ satisfies \ref{assmpn: minimal lambda for E}, then $\lambda(f)= \lambda(g)$.
\end{Lemma}

\begin{proof}
The first statement follows from \eqref{eq:compare-lambdas} since the terms involved on both sides of the equation are non-negative.
The second assertion follows from the minimality hypothesis which can be rephrased as $\lambda(f)\geq \lambda(g)$.
\end{proof}

\begin{Remark}\label{rmk:how-many-modular-forms-at-M}
The above lemma implies that if we count how often $\sum \delta(g,\ell)=0$, we can obtain a lower bound on the frequency that $\lambda(f)\leq \lambda(g)$.
Further, if \ref{assmpn: minimal lambda for E} holds, then it ensures that we can obtain a lower bound on the frequency that $\lambda(f)= \lambda(g)$.
To get a better lower bound, one would be tempted to count how frequently $\delta(g,\ell)=\delta(f,\ell)$ for all $\ell \mid M$.
However, this appears to be a hard question for the following two reasons:

Suppose that we are in the simplest setting that $\lambda(g)=0$.
In the calculations that follow, we will show how to count the number of primes for which $\delta(g,\ell)=0$.
 Let $\ell$ be a prime number that can arise as a divisor of an integer $M$ in \ref{Thm:Carayol} and suppose that $\delta(g,\ell)=0$.
Now, \cite[Theorem~A]{DT94} asserts that corresponding to each $M$ in \ref{Thm:Carayol} there is at least one newform $f$ such that $f \in \mathcal{H}(g)$.
However, as per the knowledge of the authors, it is currently not known how many $f$ such correspond to each $M$.
Therefore, it appears to be impossible to count the number of modular forms $f$ for which $\delta(f,\ell)=0$.
For results related to this question, see \cite[Proposition~5.1]{GreenbergStevens} and the introduction to \cite{Ravi}.

As a consolation, one might want to satisfy themselves by counting the number of levels $M$ in \ref{Thm:Carayol} so that there is at least one modular form $f$ satisfying the required property.
The difficulty here is that knowing the precise value of $\delta(f,\ell)$ (or $d_{\ell}(f)$) requires us to know congruence relations satisfied by $a_{\ell}(f)\pmod{p}$.
This is generally done via some matrix calculations.
The calculations we performed in Section~\ref{sec:distribution-of-fourier} allowed us count the number of primes $\ell$ with desired properties for a fixed $f$ but it does not seem to be possible to count the number of $f$ satisfying the required properties for fixed $\ell$.
\end{Remark}

A prime $\ell$ which divides the level of a modular form $f$ arising from \ref{Thm:Carayol} is of two kinds: either it divides $N$ or it does not.
We need to treat these primes separately, as we do below.

\subsubsection*{Case \emph{I}: When $\ell\mid N$}
From the discussion above, we know that
\[
\delta(g,\ell)=0 \text{ if and only if } a_\ell(g)\not\equiv \ell \pmod{p}
\]
For the remainder of this section, we make the following assumption at the finitely many primes which divide $N$.
\begin{equation}\tag*{\textup{\textbf{(Hyp~bad)}}}\label{assmpn: bad reduction of E}
\begin{minipage}{0.85\textwidth}
$d_{\ell}(g)=0$ for all primes $\ell\mid N$.
\end{minipage}
\end{equation}

This condition can be understood easily from our discussion in \S~\ref{SS: local factors}, as we now recall.
\begin{Lemma}
\label{lem:local-split-nonsplit}
Let $g \in S_2(\Gamma_0(N))$ and let $\ell$ be a prime such that $\ell \mid N$.
Then $d_{\ell}(g)=1$ if and only if either of the following two conditions hold:
\begin{enumerate}[label = \textup(\roman*\textup)]
\item $a_\ell(g) =1$  and $\ell \equiv 1 \pmod{p}$.
\item $a_\ell(g) = -1$ and $\ell \equiv -1 \pmod{p}$.
\end{enumerate}
\end{Lemma}
\begin{proof}
This follows immediately from \eqref{eqn:local-at-bad-primes}.
\end{proof}

\begin{Remark}
\label{rem:local-split-nonsplit ec}
When $g$ corresponds to an elliptic curve, this condition can be restated in terms of the reduction type at a prime $\ell\mid N$.
In this case, $d_{\ell}(g)=1$ if and only if either of the following two conditions hold:
\begin{enumerate}[label = \textup(\roman*\textup)]
\item $\ell$ is a prime of split multiplicative reduction and $\ell \equiv 1 \pmod{p}$.
\item $\ell$ is a prime of non-split multiplicative reduction and $\ell \equiv -1 \pmod{p}$.
\end{enumerate}
\end{Remark}

\subsubsection*{Case \emph{II}: When $\ell\nmid N$}
In this case, we know from the discussion above that $\delta(g,\ell)=0$ if and only if $a_\ell(g) \not\equiv \ell + 1 \pmod{p}$.
We can divide this into two cases:
\begin{enumerate}
\item If $\ell \equiv -1 \pmod{p}$, then $\delta(g,\ell) =0$ provided $a_\ell(g) \not\equiv 0 \pmod{p}$.
But these are precisely the primes in Set~2', which is a \emph{discard set}, so these primes do not arise in our count.
\item If $\ell \not\equiv -1 \pmod{p}$, then we must count the primes from $\mathfrak{S}_1$ for which $a_\ell(g) \equiv -(\ell+1)\pmod{p}$ as well as the primes from $\mathfrak{S}_3$ for which $a_\ell(g) \not\equiv 2\pmod{p}$.
\end{enumerate}
Let $\mathfrak{S}_{1,\delta=0}$ denote the subset of primes in $\mathfrak{S}_1$ satisfying the additional property that $\delta(g,\ell)=0$.

\begin{Lemma}
\label{eq:box-1-delta-zero}
With notation introduced above, the primes in $\mathfrak{S}_{1,\delta=0}$ have density
\[
\mathfrak{d}(\mathfrak{S}_{1,\delta=0}) = \frac{(p-3)}{(p-1)^2}.
\]
Moreover, if $p>3$ then
\[
\mathfrak{d_1} := \frac{\mathfrak{d}(\mathfrak{S}_{1,\delta=0})}{(\mathfrak{d}(\mathfrak{S}_{1})} = \frac{1}{2}.
\]

\end{Lemma}

\begin{proof}
To first assertion follows from the above discussion and repeating the same argument used to establish Lemma~\ref{box 1 count}.
In particular,
\[
\mathfrak{d}(\mathfrak{S}_{1,\delta=0}) =\left(\sum_{\ell=2}^{p-2} \frac{\# C_{\ell,-(\ell+1)}}{\# \GL_2(\F_p)}\right) \pi(x) = \frac{(p-3)}{(p-1)^2}.
\]
The second assertion is immediate.
\end{proof}

The above calculation can be summarized as follows:
let $\ell$ be a prime in $\mathfrak{S}_{1,\delta=0}$; this is a \emph{positive density} subset of $\mathfrak{S}_1$ (when $p>3$).
By \ref{Thm:Carayol} and \cite[Theorem~A]{DT94}, the prime $\ell$ divides (infinitely many) integers $M$, each of which is the level of at least one modular form $f$ belonging to $\mathcal{H}(g)$.
However, this prime $\ell$ does not contribute to the $\lambda$-invariant growth of $f$ (relative to the $\lambda$-invariant of $g$).
Now we compute the density of primes from $\mathfrak{S}_3$ for which $\delta(g,\ell)=0$.
Denote this set by $\mathfrak{S}_{3,\delta=0}$.
For this, we must consider the primes $\ell \leq x$ such that $\ell \equiv 1 \pmod{p}$ and $a_\ell(g)\not\equiv 2 \pmod{p}$.
By \cite[Proposition~4.6]{G-FP} this density is given by
\begin{equation}
\label{eq:box-3-delta-zero}
\mathfrak{d}(\mathfrak{S}_{3,\delta=0}) = \frac{(p^2-p-1)}{(p-1)(p^2-1)}.
\end{equation}
Thus, analogous to what we have seen above, there is a subset of $\mathfrak{S}_3$ with density $\mathfrak{d}_3= \frac{p^2-p-1}{p^2-1}$ such that any prime $\ell$ in this subset provides no local contribution to the (relative) $\lambda$-invariant growth of a modular form $f \in \mathcal{H}(g)$.

By Lemma~\ref{eq:box-1-delta-zero} and \eqref{eq:box-3-delta-zero}, we find that the density of primes $\ell$ appearing in \ref{Thm:Carayol} for which $\delta(g,\ell)=0$ is given by
\[
\frac{(p-3)}{(p-1)^2} + \frac{(p^2-p-1)}{(p-1)(p^2-1)} = \frac{(2p^2-3p-4)}{(p-1)^2(p+1)}.
\]
We can now state and prove the main theorem of this section.

\begin{Th}
\label{thm: reformulation of lambda f equal lambda E case}
Let $g \in S_2(\Gamma_0(N))$ be a newform satisfying \ref{assmpn: optimal}.
Suppose that $\mu(g)=0$ and that assumptions \ref{assmpn: minimal lambda for E} and \ref{assmpn: bad reduction of E} are satisfied.
Then, there exists a set of prime numbers $\mathcal{R}_1$ of density $\frac{(p-3)}{(p-1)^2}$ and another set of prime numbers $\mathcal{R}_2$ of density $\frac{p^2-p-1}{(p-1)(p^2-1)}$ such that for every integer $M$ of the form
\begin{equation}
\label{eqn: form of M}
M = N \prod_{\ell_{1}\in \mathcal{R}_1} \ell_{1}^{\alpha(\ell_1)}\prod_{\ell_{2}\in \mathcal{R}_2} \ell_{2}^{\alpha(\ell_2)} \text{ where } 0\leq \alpha(\ell_1)\leq 1 \text{ and } 0\leq \alpha(\ell_2)\leq 2,
\end{equation}
there exists a modular form $f \in \mathcal{H}(g)$ of level $M$ with $\lambda(g)=\lambda(f)$.
\end{Th}

\begin{proof}
Since we assume that \ref{assmpn: bad reduction of E} holds, for every prime $\ell\mid N$ we have that $\delta(g,\ell)=0$.

Choose $\mathcal{R}_1$ (resp. $\mathcal{R}_2$) to be the set of primes $\mathfrak{S}_{1,\delta=0}$ (resp. $\mathfrak{S}_{3,\delta=0}$) described above.
Our calculations above show that
\begin{align*}\frak{d}(\mathcal{R}_1)= \mathfrak{d}(\mathfrak{S}_{1,\delta=0})  &= \frac{p-3}{(p-1)^2}\\
\frak{d}(\mathcal{R}_2) = \mathfrak{d}(\mathfrak{S}_{3,\delta=0}) &= \frac{p^2-p-1}{(p-1)(p^2-1)}.
\end{align*}
Since the chosen set $\mathcal{R}_1$ (resp. $\mathcal{R}_2$) is a subset of $\mathfrak{S}_1$ (resp. $\mathfrak{S}_3$), we know from the discussion in Section~\ref{reformulating carayol} combined with \cite[Theorem~A]{DT94} that every integer of the form \eqref{eqn: form of M} arises as the level of a modular form $f \in \mathcal{H}(g)$.
By construction, each $\ell$ which is a divisor of $\frac{M}{N}$ satisfies the condition that $\delta(g,\ell)=0$.
The result follows from Lemma~\ref{lemma: delta equals 0 implies lower bound}.
\end{proof}

We will now illustrate the above result with a concrete example.
\begin{Example}
\label{Example: 11a1}
Consider the modular form $g_\EC$ associated to the elliptic curve $\EC=$\href{https://www.lmfdb.org/EllipticCurve/Q/11/a/2}{$11a1$} (Cremona label).
We know that $\EC$ has Mordell-Weil rank 0.
\cite[Theorem~3.7]{KR21} asserts that for $100\%$ of primes $p$, the corresponding Iwasawa invariants, namely $\lambda(g_\EC)= \mu(g_E)=0$.
So, \ref{assmpn: minimal lambda for E} is satisfied for $100\%$ of the primes $p$.

Note that $\EC$ has conductor $N=11$ and $\ell=11$ is a prime of split multiplicative reduction.
By Lemma~\ref{lem:local-split-nonsplit}, for all $p>11$ we have that $d_{\ell}(g_\EC)=0$.
In particular, \ref{assmpn: bad reduction of E} holds for all $p>11$.
In fact, the only primes for which \ref{assmpn: bad reduction of E} does not hold are $p=2$ and $5$; furthermore, when $p=2$, $\EC$ has good supersingular reduction, and when $p=5$ it is known that $\mu(g_\EC)=1$.

Thus, Theorem~\ref{thm: reformulation of lambda f equal lambda E case} holds for $100\%$ of primes $p$.
\end{Example}




\subsection{How often is \texorpdfstring{$\lambda(f)>\lambda(g)$}{}?}
In this section, we use \eqref{eqn: strict inequality} to study the frequency with which the $\lambda$-invariant is guaranteed to increase for $f \in \mathcal{H}(g)$.
Thus, contrary to our strategy in \S~\ref{subsec:equal-lambda}, we instead seek to maximize the local factors $\delta(g,\ell)$.

We begin by recording a useful lemma.
\begin{Lemma}
\label{useful lemma positive case}
Let $g \in S_2(\Gamma_0(N))$ be a newform satisfying \ref{assmpn: optimal}.
Let $M$ be an integer arising from \ref{Thm:Carayol} and $f\in \mathcal{H}(g)$ a modular form of level $M$.
If $\delta(g,\ell) =1$ for every $\ell\mid N$ and $\delta(g,\ell) =2$ for every $\ell\mid \frac{M}{N}$, then $\lambda(f)> \lambda(g)$.
\end{Lemma}

\begin{proof}
If $\delta(g,\ell) =1$ for every $\ell\mid N$ and $\delta(g,\ell) =2$ for every $\ell\mid \frac{M}{N}$, then
\[
\sum_{\ell\mid M}\delta(g,\ell) > \sum_{\ell\mid M}\delta(f,\ell).
\]
The lemma is immediate from \eqref{eqn: strict inequality}.
\end{proof}

For the remainder of this section, we consider modular forms $g \in S_2(\Gamma_0(N))$  satisfying the following condition:
\begin{equation}\tag*{\textup{\textbf{(Hyp~bad')}}}\label{assmpn: d=1 case}
\begin{minipage}{0.8\textwidth}
$d_{\ell}(g)=1$ for all primes $\ell\mid N$.
\end{minipage}
\end{equation}

Next, we focus on the primes $\ell\nmid N$.
We know from the discussion about simple roots in Section~\ref{SS: local factors} that $d_{\ell}(g)=2$ precisely when both of the following two conditions hold:
\begin{enumerate}
\item $\ell \equiv 1 \pmod{p}$.
\item $a_\ell(g) \equiv 1 + \ell \equiv 2 \pmod{p}$
\end{enumerate}
In Lemma~\ref{box 3 count}, we computed the density of primes in the set $\mathfrak{S}_3$, i.e., primes $\ell$ such that $\ell\equiv 1\pmod{p}$.
This is given by
\[
\frak{d}(\mathfrak{S}_3)=\frac{\pi(x)}{p-1}.
\]
Let us denote the subset of primes in $\mathfrak{S}_3$ for which $a_{\ell}\equiv 2\pmod{p}$ by $\mathfrak{S}_{3,d=2}(x)$.
The set $\mathfrak{S}_3$ can be split into two disjoints subsets; namely one where $a_{\ell}\not\equiv 2\pmod{2}$ and one where $a_{\ell}\equiv 2\pmod{2}$.
It follows from \eqref{eq:box-3-delta-zero} that
\begin{align*}
\frak{d}(\mathfrak{S}_{3,d=2}) :=\frak{d}(\mathfrak{S}_3) - \frak{d}(\mathfrak{S}_{3,\delta=0}) &= \frac{1}{p-1} - \frac{(p^2-p-1)}{(p-1)(p^2-1)}\\
& = \frac{p}{(p-1)(p^2-1)}.
\end{align*}
It is easy to see from the above discussion that there is a subset of $\mathfrak{S}_3$ with density $\mathfrak{d}'_3= \frac{p}{p^2-1}$ such that any prime $\ell$ in this subset provides \emph{positive} local contribution to the (relative) $\lambda$-invariant growth of a modular form $f\in \mathcal{H}(g)$.

We now state the main result of this section.
Since the proof is analogous to that of Theorem~\ref{thm: reformulation of lambda f equal lambda E case} we only provide a brief sketch of the proof.
\begin{Th}
\label{thm: reformulation of lambda positive}
Let $g \in S_2(\Gamma_0(N))$ be a newform satisfying \ref{assmpn: optimal} and \ref{assmpn: d=1 case}.
Assume that $\mu(g)=0$.
There exists a set of primes $\mathcal{R}$ of density $\frac{p}{(p-1)(p^2 -1)}$ such that for every integer $M$ of the form
\begin{equation}
\label{eqn: form of M lambda positive}
M = N\prod_{\ell\in \mathcal{R}} \ell^{\alpha(\ell)} \text{ where } 0\leq \alpha(\ell)\leq 2,
\end{equation}
there exists a modular form $f\in\mathcal{H}(g)$ of level $M$ with $\lambda(f)>\lambda(g)$.
\end{Th}

\begin{proof}
\ref{assmpn: d=1 case} implies that for every prime $\ell\mid N$, we have $\delta(g,\ell)=1$.

Next, we choose $\mathcal{R}$ to be the set of primes $\mathfrak{S}_{3,d=2}$ satisfying
\[\frak{d}(\mathfrak{S}_{3,d=2}) = \frac{p}{(p-1)(p^2-1)}.
\]
By \cite[Theorem~A]{DT94}, every integer of the form \eqref{eqn: form of M lambda positive} is the level of a modular form $f \in \mathcal{H}(g)$.
Each $\ell\mid \frac{M}{N}$ satisfies the additional condition that $\delta(g,\ell)=2$.
The result follows from Lemma~\ref{useful lemma positive case}.
\end{proof}


We now illustrate our results with concrete examples

\begin{Example}
Let $g_\EC$ be the modular form associated  to the elliptic curve $\EC$ with Cremona label \href{https://www.lmfdb.org/EllipticCurve/Q/43/a/1}{43a1}.
Fix $p=11$; this is a prime of good ordinary reduction of $\EC$ with $\mu(g_E)=0$.

We have that $g$ satisfies \ref{assmpn: optimal}, and since $43$ is prime, $g$ is of optimal level in $\mathcal{H}(g)$.

The prime $\ell=43$ is a prime of non-split multiplicative reduction for $\EC$, so $a_{43}(g)=-1$.
Since $43 \equiv -1 \pmod{11}$, it follows from Lemma~\ref{lem:local-split-nonsplit} that $d_{43}(\EC) = 1$.

Hence Theorem~\ref{thm: reformulation of lambda positive} holds in this case, and the density of $\mathcal{R}$ is $\frac{11}{1200}\sim 0.00917$.
\end{Example}

\begin{Example}
The results of this paper can be combined with recent results regarding analytic rank parity; see e.g. \cite{Shekhar,Hat17,AAS,hatleylei2019}.
We illustrate this with a concrete example.

Let $g_\EC$ be the modular form associated to the elliptic curve $\EC$ with Cremona label \href{https://www.lmfdb.org/EllipticCurve/Q/53/a/1}{53a1}.
The analytic rank of $g_\EC$ is $1$.
Let $p=3$; then $\EC$ has good supersingular reduction at $p$, and one checks that \ref{assmpn: optimal} and \ref{assmpn: d=1 case} hold.
The LMFDB tells us $\mu^\pm(\EC)=0$.
Hence Theorem~\ref{thm: reformulation of lambda positive} holds in this case, and the density of $\mathcal{R}$ is $\frac{3}{16}=0.1875$.
We will now apply \cite[Corollary~5.8]{hatleylei2019}.
Let $f \in \mathcal{H}(g)$ have level $M$, where $M$ is formed entirely from the set of primes guaranteed by Theorem~\ref{thm: reformulation of lambda positive}.
Denote by $\mathfrak{m}$ the number of primes $\ell$ such that $\ell \mid \frac{M}{N}$.
Note that $\delta(f,53)$ is either $0$ or $1$, depending on the value of $a_{53}(f)$.

In the notation of \cite[Corollary~5.8]{hatleylei2019}, we then have $r_{\mathrm{an}}(g)=1$ and $|S_g|=1$, while $|S_f|= \mathfrak{m} + \delta(f,53)$.
So by \cite[Corollary~5.8]{hatleylei2019}, we see that the analytic rank of $f$ has the same parity as the quantity $\mathfrak{m} + \delta(f,53)$.
\end{Example}

\begin{Remark}
We make one final remark regarding the modular forms $f$ which are under study in this paper.
By definition, the modular forms in $\mathcal{H}(g)$ are of weight $2$, and we can determine their levels using \ref{Thm:Carayol}.

However, there also exist modular forms $f$ of higher weight $k \equiv 2 \mod (p-1)$ which satisfy $\bar{\rho}_f \simeq \bar{\rho}_g$.
If $g$ is $p$-ordinary, then the same is true for $f$, and in fact $f$ lives in a \textit{Hida family} of modular forms with residual representations isomorphic to $\bar{\rho}_g$; by \cite[Corollary 2]{epw}, all of these modular forms have the same $\lambda$-invariant (assuming \ref{assmpn: mu}).
In particular, this branch of the Hida family contains a weight $2$ modular form which belongs to $\mathcal{H}(g)$, and since its $\lambda$-invariant is the same as $f$, we may propagate our results to higher-weight modular forms in this way.

\end{Remark}

\section{Declarations}

The authors have no relevant financial or non-financial interests to disclose.

\section{Data Availability}

Data sharing is not applicable to this article as no datasets were generated or analysed during the current study.

\bibliographystyle{amsalpha}
\bibliography{references}

\providecommand{\bysame}{\leavevmode\hbox to3em{\hrulefill}\thinspace}
\providecommand{\MR}{\relax\ifhmode\unskip\space\fi MR }
\providecommand{\MRhref}[2]{%
  \href{http://www.ams.org/mathscinet-getitem?mr=#1}{#2}
}
\providecommand{\href}[2]{#2}
\begin{thebibliography}{EPW06}

\bibitem[AAS17]{AAS}
Suman Ahmed, Chandrakant Aribam, and Sudhanshu Shekhar, \emph{Root numbers and
  parity of local {I}wasawa invariants}, J. Number Theory \textbf{177} (2017),
  285--306. \MR{3629245}

\bibitem[BCP97]{magma}
Wieb Bosma, John Cannon, and Catherine Playoust, \emph{The {M}agma algebra
  system. {I}. {T}he user language}, J. Symbolic Comput. \textbf{24} (1997),
  no.~3-4, 235--265, Computational algebra and number theory (London, 1993).
  \MR{MR1484478}

\bibitem[Car89]{Car89}
Henri Carayol, \emph{Sur les repr{\'e}sentations {G}aloisiennes modulo $l$
  attach{\'e}es aux formes modulaires}, Duke Math. J \textbf{59} (1989), no.~3,
  785--801.

\bibitem[Del71]{deligne}
Pierre Deligne, \emph{Formes modulaires et repr\'{e}sentations {$l$}-adiques},
  S\'{e}minaire {B}ourbaki. {V}ol. 1968/69: {E}xpos\'{e}s 347--363, Lecture
  Notes in Math., vol. 175, Springer, Berlin, 1971, pp.~Exp. No. 355, 139--172.
  \MR{3077124}

\bibitem[DS05]{DS05}
Fred Diamond and Jerry~Michael Shurman, \emph{A first course in modular forms},
  vol. 228, Springer, 2005.

\bibitem[DT94a]{DT2}
Fred Diamond and Richard Taylor, \emph{Lifting modular mod {$l$}
  representations}, Duke Math. J. \textbf{74} (1994), no.~2, 253--269.
  \MR{1272977}

\bibitem[DT94b]{DT94}
\bysame, \emph{Non-optimal levels of $\mod l$ modular representations}, Invent.
  math. \textbf{115} (1994), no.~1, 435--462.

\bibitem[EPW06]{epw}
Matthew Emerton, Robert Pollack, and Tom Weston, \emph{Variation of {I}wasawa
  invariants in {H}ida families}, Invent. Math. (2006), no.~163, 523--580.

\bibitem[GFP20]{G-FP}
Natalia Garcia-Fritz and Hector Pasten, \emph{Towards {H}ilbert's tenth problem
  for rings of integers through {I}wasawa theory and {H}eegner points}, Math.
  Ann. \textbf{377} (2020), no.~3-4, 989--1013. \MR{4126887}

\bibitem[Gre89]{Gre89}
Ralph Greenberg, \emph{Iwasawa theory for {$p$}-adic representations},
  Algebraic number theory, Adv. Stud. Pure Math., vol.~17, Academic Press,
  Boston, MA, 1989, pp.~97--137.

\bibitem[Gre99]{Greenberg}
\bysame, \emph{Iwasawa theory for elliptic curves}, Arithmetic theory of
  elliptic curves ({C}etraro, 1997), Lecture Notes in Math., vol. 1716,
  Springer, Berlin, 1999, pp.~51--144.

\bibitem[GS94]{GreenbergStevens}
Ralph Greenberg and Glenn Stevens, \emph{On the conjecture of {M}azur, {T}ate,
  and {T}eitelbaum}, {$p$}-adic monodromy and the {B}irch and
  {S}winnerton-{D}yer conjecture ({B}oston, {MA}, 1991), Contemp. Math., vol.
  165, Amer. Math. Soc., Providence, RI, 1994, pp.~183--211. \MR{1279610}

\bibitem[GV00]{greenbergvatsal}
Ralph Greenberg and Vinayak Vatsal, \emph{On the {I}wasawa invariants of
  elliptic curves}, Invent. Math. \textbf{142} (2000), no.~1, 17--63.

\bibitem[Hat17]{Hat17}
Jeffrey Hatley, \emph{Rank parity for congruent supersingular elliptic curves},
  Proc. Am. Math. Soc. \textbf{145} (2017), no.~9, 3775--3786.

\bibitem[HL19]{hatleylei2019}
Jeffrey Hatley and Antonio Lei, \emph{Arithmetic properties of signed {Selmer}
  groups at non-ordinary primes}, Annales de l'Institut Fourier \textbf{69}
  (2019), no.~3, 1259--1294.

\bibitem[HL21]{hatley-lei-jnt}
\bysame, \emph{Comparing anticyclotomic {S}elmer groups of positive coranks for
  congruent modular forms---{P}art {II}}, J. Number Theory \textbf{229} (2021),
  342--363. \MR{4279284}

\bibitem[Kat04]{Kato}
Kazuya Kato, \emph{{$p$}-adic {H}odge theory and values of zeta functions of
  modular forms}, Ast\'{e}risque (2004), no.~295, ix, 117--290, Cohomologies
  $p$-adiques et applications arithm\'{e}tiques. III.

\bibitem[Kim09]{kim09}
Byoung~Du Kim, \emph{The {I}wasawa invariants of the plus/minus {S}elmer
  groups}, Asian J. Math. \textbf{13} (2009), no.~2, 181--190.

\bibitem[Kim13]{kim13}
\bysame, \emph{The plus/minus {S}elmer groups for supersingular primes}, J.
  Australian Math. Soc. \textbf{95} (2013), no.~2, 189--200.

\bibitem[Kob03]{kobayashi03}
Shin-ichi Kobayashi, \emph{Iwasawa theory for elliptic curves at supersingular
  primes}, Invent. Math. \textbf{152} (2003), no.~1, 1--36.

\bibitem[KR21]{KR21}
Debanjana Kundu and Anwesh Ray, \emph{Statistics for {I}wasawa invariants of
  elliptic curves}, Trans. Am. Math. Soc. \textbf{374} (2021), 7945--7965, DOI:
  https://doi.org/10.1090/tran/8478.

\bibitem[Lei11]{lei09}
Antonio Lei, \emph{Iwasawa theory for modular forms at supersingular primes},
  Compositio Math. \textbf{147} (2011), no.~03, 803--838.

\bibitem[Loe17]{LoefflerGlasgow}
David Loeffler, \emph{Images of adelic {G}alois representations for modular
  forms}, Glasg. Math. J. \textbf{59} (2017), no.~1, 11--25.

\bibitem[LW12]{LW}
David Loeffler and Jared Weinstein, \emph{On the computation of local
  components of a newform}, Math. Comp. \textbf{81} (2012), no.~278,
  1179--1200.

\bibitem[Mat07]{matsuno2007}
Kazuo Matsuno, \emph{Construction of elliptic curves with large {I}wasawa
  {$\lambda$}-invariants and large {T}ate-{S}hafarevich groups}, Manuscripta
  Math. \textbf{122} (2007), no.~3, 289--304. \MR{2305419}

\bibitem[Ram14]{Ravi}
Ravi Ramakrishna, \emph{Maps to weight space in {H}ida families}, Indian J.
  Pure Appl. Math. \textbf{45} (2014), no.~5, 759--776. \MR{3286085}

\bibitem[Ray23]{Ray-lambda}
Anwesh Ray, \emph{Constructing {G}alois representations with large {I}wasawa
  {$\lambda$}-invariant}, Ann. Math. Qu\'{e}bec (2023),
  https://doi.org/10.1007/s40316-023-00212-5.

\bibitem[She16]{Shekhar}
Sudhanshu Shekhar, \emph{Parity of ranks of elliptic curves with equivalent
  {${\rm mod}\,p$} {G}alois representations}, Proc. Amer. Math. Soc.
  \textbf{144} (2016), no.~8, 3255--3266. \MR{3503694}

\bibitem[Spr12]{sprung09}
Florian E.~Ito Sprung, \emph{Iwasawa theory for elliptic curves at
  supersingular primes: a pair of main conjectures}, J. Number Theory
  \textbf{132} (2012), no.~7, 1483--1506.

\bibitem[Stu87]{sturm}
Jacob Sturm, \emph{On the congruence of modular forms}, Number theory ({N}ew
  {Y}ork, 1984--1985), Lecture Notes in Math., vol. 1240, Springer, Berlin,
  1987, pp.~275--280. \MR{894516}

\end{thebibliography}

\end{document}